\documentclass[a4paper,12pt]{article}
\usepackage{amssymb}
\usepackage{latexsym}
\usepackage{amsmath,amsthm,amssymb}
\usepackage{amsfonts}

\newtheorem{tw}{Theorem}[section]

\newtheorem{pro}[tw]{Proposition}
\newtheorem{cor}[tw]{Corollary}

\theoremstyle{definition}

\newtheorem{exa}[tw]{Example}
\newtheorem{rem}[tw]{Remark}

\begin{document}

\begin{center}

{\Large A note on the gambling team method}\\
{\sc Krzysztof Zajkowski}\\
Institute of Mathematics, University of Bialystok \\ 
Akademicka 2, 15-267 Bialystok, Poland \\ 
E-mail:kryza@math.uwb.edu.pl 
\end{center}
\begin{abstract}

Gerber and Li in \cite{GeLi} formulated, using a Markov chain embedding, a system of equations that describes relations between generating functions of waiting time distributions for occurrences of patterns  in a sequence of independent repeated experiments when initial outcomes of the process are known. We show how this system of equations can be obtained by using the classical gambling team technique . 
We also present a form of solution of the system and give an example showing how first
results of  trials  influence the probabilities that a chosen pattern precedes remaining ones in a realization of the process. 
 
\end{abstract}

{\it 2010 Mathematics Subject Classification:} 60G40, 60G42. 

{\it Key words: martingales, stopping times, optional stopping theorem, gambling team technique, generating functions}

\section{Introduction and notation}

In  a study of the occurrence of patterns in a process of independent repeated experiments Li in \cite{Li} invented a martingale method (the gambling team method)
and formed a system of equations which related the expected time of waiting  until  any pattern appears and the probabilities that one of patterns
precedes  the remaining ones. Gerber and Li \cite{GeLi} using a Markov chain embedding extended this result to generating functions of waiting time distributions for patterns. The main goal is to show how immediately, using the classical gambling team method, one can obtain the system of equations for the generating functions and also to propose some form of solution of this system.

A development of the gambling team method to many teams of gamblers was introduced by Pozdnyakov et al. \cite{PozGKS}
(see also \cite{PozKull}). It can be used to compute higher moments, generating functions of the waiting time and to calculate probabilities for scan statistics (see \cite{PozGKS,PozStee}). In \cite{PozGKS1}, \cite{Poz} and \cite{FisCui} one can find an application of the method of gambling team to investigations of occurrences of patterns in Markov chains. A more general technique for the Markov chain embedding method was introduced by Fu \cite{Fu}, and has been further developed by other authors
(see \cite{FuLou} for more details).

Throughout the article we employ the following notation. Let $\xi$ be an arbitrary but fixed discrete random variable. We  call the set $\Sigma$ of possible values of $\xi$  the alphabet. We assume that the probability of each letter is positive: $Pr(\xi=a)>0$ for any $a\in\Sigma$.
Let $(\xi_n)_{n=1}^\infty$ be a sequence of independent, identically distributed random letters over $\Sigma$ having the same distribution as $\xi$.

By a pattern (word) $B$ of the length $m$ we mean a finite ordered sequence of letters $b_1b_2...b_m$. 
Let $\tau_B$ denote a time of waiting (stopping) until $B$ occurs as a run in the process $(\xi_n)$. 
We  assume  that a pattern $A=a_1...a_l$ is already given at the beginning of the process and that $B$ is not a subpattern of $a_1...a_{l-1}$.
Define now a stopping time for $B$, given $A$ to start with: 
\begin{equation}
\label{taui}
\tau_{AB}=\min\{k\ge 0:\;B \;{\rm is\;a\;subpattern\; of}\;a_1...a_{l-1}\xi_0\xi_1...\xi_k\}; 
\end{equation}
we assume that $\xi_0=a_l$.
\newline
First we are interested in the expected waiting time of $\tau_{AB}$.
We recall a general solution, obtained by  Gerber and Li \cite{GeLi}, based on an application of martingale techniques to a derivation of the probability generating function 
of $\tau_{AB}$.

Before we show this solution, we introduce some more notation. For a given pattern $A=a_1...a_l$ writing $Pr(A)$ we mean the product of probabilities $Pr(\xi=a_1)\cdot...\cdot Pr(\xi=a_l)$.
Let $A_{(k)}$ and $A^{(k)}$  denote subpatterns  formed by first and last $k$ letters 
of $A$, respectively; i.e. $A_{(k)}=a_1a_2...a_k$ and $A^{(k)}=a_{l-k+1}a_{l-k+2}...a_l$.  For patterns $A$ and $B$ we adopt the notation 
$[A^{(k)} = B_{(k)}]=1$ if $A^{(k)} = B_{(k)}$  and $[A^{(k)} = B_{(k)}]=0$ if not (roughly speaking the square bracket takes logical values from a sentence contained in it).
Let $0<\alpha <1$. 
We define a correlation function $(A\ast B)(\alpha)$  as
$$
(A\ast B)(\alpha)= \sum_{k=1}^{\min\{l,m\}}\frac{[A^{(k)} = B_{(k)} ]}{Pr(B_{(k)})\alpha^{k}}
=\sum_{k=1}^{\min\{l,m\}}\frac{[A^{(k)} = B_{(k)}]}{Pr(\xi=b_1)\cdot...\cdot Pr(\xi=b_k)\alpha^k}.
$$

\begin{exa}
Let $\Sigma=\{H,T\}$. Assume that 
$$
Pr(\xi=H)=p\quad{\rm and}\quad Pr(\xi=T)=q=1-p.
$$
Consider two patterns $A=THH$ and $B=THTH$. Then the correlation functions have the following forms:
\begin{eqnarray*}
(A\ast A)(\alpha)=\frac{1}{p^2q\alpha^3},&\;& (B\ast A)(\alpha)=\frac{1}{pq\alpha^2},\\ 
(A\ast B)(\alpha)=0,& \; & (B\ast B)(\alpha)=\frac{1}{pq\alpha^2}+\frac{1}{p^2q^2\alpha^4}. 
\end{eqnarray*}
\end{exa} 

\section{Gambling team technique}
\label{Gteam}
Let a casino generate the sequence of letters $(\xi_n)$. We wait for the information on a realization of the pattern $B$. We are impatient and 
visit the casino at the $l$th moment (right after the $l$th round) and observe that at the beginning the pattern $A$ occurs. We ask how long, on average, we would wait for  $B$.

Consider a flow of gamblers (a gambling team) visiting the casino.  Let $0<\alpha<1$ and  the $n$th gambler arrives right before $n$th round  and places \$$\alpha^{n-1}$ bet that
$\xi_n=b_1$. If $\xi_n$ is not $b_1$ the gambler loses the bet and leaves the game.
If $\xi_n=b_1$ the casino pays fair odds 
$$
\frac{\alpha^{n-1}}{Pr(\xi_n=b_1)}=\frac{\alpha^{n-1}}{Pr(\xi=b_1)}.  
$$
Next the gambler bets his entire capital on $\xi_{n+1}=b_2$. If it is not $b_2$, he/she goes home
with nothing; otherwise he/she increases his/her capital by the factor $1/Pr(\xi=b_2)$. 
Then he/she continues in the same fashion until the entire
word $B$ is exhausted. If the gambler is lucky he/she leaves the game with total winnings of
$$
\frac{\alpha^{n-1}}{Pr(\xi_n=b_1)\cdot...\cdot Pr(\xi_{n+m-1}=b_m)}=\frac{\alpha^{n-1}}{Pr(\xi=b_1)\cdot...\cdot Pr(\xi=b_m)}=\frac{\alpha^{n-1}}{Pr(B)}
$$
dollars. 
We should remember that in the meantime new players entered  the game and may
also have some amounts of money. It depends on whether some initial parts of $B$ overlap with final ones.

Let $X_n$ denote the total net gain of the casino at the moment $n$. Let us emphasize that $X_0=0$ and $X_n$ is bounded from above by the entire capital of gambling team that is $X_n\le \frac{1}{1-\alpha}$. Under a theoretical assumption that each gambler is 'lucky' and wins maximal amount of money one can bound $X_n$ from below by $\frac{-1}{(1-\alpha)Pr(B)}$. Thus  we get
$$
\vert X_n \vert \le \frac{1}{(1-\alpha)Pr(B)}.
$$
Let us recall that we know the first $l$ generated letters (pattern $A$). For this reason
we can calculate the deterministic value of $X_l$:
\begin{equation}
\label{ixm}
X_l=1+\alpha+\ldots+\alpha^{l-1}-\alpha^l(A\ast B)(\alpha)=\frac{1-\alpha^l}{1-\alpha}-\alpha^l(A\ast B)(\alpha),
\end{equation}
where first summands are payments of  first $l$ gamblers and $\alpha^l(A\ast B)(\alpha)$ gives their winnings at the moment $l$. Let $Y_n=X_{l+n}$.
Note that $Y_0=X_l$.
Because the game is fair in each round, the sequence $(Y_n)_{n=0}^\infty$ forms a martingale. It is bounded. So for any stopping time $\tau$, by the optional
stopping theorem, we have 
\begin{equation}
\label{opt}
Y_0=EY_\tau.
\end{equation}
In the following let $\tau=\tau_{AB}$. Then
\begin{equation}
\label{stop}
Y_\tau=X_{l+\tau}= 1+\alpha+\ldots+\alpha^{l+\tau-1}-\alpha^{l+\tau}(B\ast B)(\alpha)=\frac{1-\alpha^{l+\tau}}{1-\alpha}-\alpha^{l+\tau}(B\ast B)(\alpha),
\end{equation}
where $\alpha^{l+\tau}(B\ast B)(\alpha)$ is the total winning of the gambling team by time $\tau=\tau_{AB}$.
By virtue of (\ref{ixm}),(\ref{opt}) and (\ref{stop}) we obtain
$$
\frac{1-\alpha^lE(\alpha^\tau)}{1-\alpha}-\alpha^l(B\ast B)(\alpha)E(\alpha^\tau)=\frac{1-\alpha^l}{1-\alpha}-\alpha^l(A\ast B)(\alpha)
$$
and hence we get the formula for the probability generating function of $\tau_{AB}$
$$
E(\alpha^{\tau_{AB}})=\frac{1+(1-\alpha)(A\ast B)(\alpha)}{1+(1-\alpha)(B\ast B)(\alpha)};
$$
compare Theorem 4.1 in \cite{GeLi}.

Let $g_\tau(\alpha)$ denote $E(\alpha^{\tau})$ and $Q_\tau(\alpha)$ a function $\frac{1-g_\tau(\alpha)}{1-\alpha}$, which is 
the generating function of the sequence of cumulative probabilities $(Pr(\tau>n))_{n=0}^\infty$ that is
$\sum_{n=0}^\infty Pr(\tau>n)\alpha^n$. If the value $Q_\tau(1)$, as the limit for $\alpha$ tending to 1, exists, then $E\tau=Q_\tau(1)$.
In our case 
$$
Q_{\tau_{AB}}(\alpha)=\frac{(B\ast B)(\alpha)-(A\ast B)(\alpha)}{1+(1-\alpha)(B\ast B)(\alpha)}
$$
and hence
$$
E\tau_{AB}=Q_{\tau_{AB}}(1)=(B\ast B)(1)-(A\ast B)(1).
$$

\section{Competing patterns}
Let $\mathfrak{B}$ denote a collection of $m$ patterns (words)  $B_i$ ($1\le i \le m$) of lengths $l_i$, respectively. 
We assume that none of them contains any other as a subpattern.
As before we  consider the situation where a pattern $A=a_1...a_l$  is given at the beginning of the process and none $B_i$ is a subpattern
of $a_1...a_{l-1}$.

Let now $\tau_{A\mathfrak{B}}$ be the time of stopping  until one of the collection of patterns $\mathfrak{B}$ is observed, given $A$ to start with, i.e. 
\begin{equation}
\label{stoptime}
\tau_{A\mathfrak{B}}=\min\{\tau_{AB_i}:\;1\le i \le m\}.
\end{equation}
Assuming that the gambling team bets on the chosen pattern $B_i$ ($1\le i \le m$)
and using
instead  $\tau_{AB_i}$ the stopping time $\tau_{A\mathcal{B}}$,
we obtain a  system of equations equivalent to (39) \cite{GeLi}.

\begin{pro}
\label{gamb}
Let $\tau$ denote the stopping time $\tau_{A\mathfrak{B}}$ defined as above and $\tau_i$ the stopping time $\tau_{AB_i}$ defined by (\ref{taui}).
Let  $g^j_\tau(\alpha)$ be the function $E(\alpha^\tau{\bf 1}_{\{\tau=\tau_j\}})$, where ${\bf 1}_{\{\tau=\tau_j\}}$ is the indicator function of the event ${\{\tau=\tau_j\}}$. Then for every $i$ ($1\le i \le m$) the following equation holds
$$
-Q_\tau(\alpha) +\sum_{j=1}^m (B_j\ast B_i)(\alpha)g_\tau^{j}(\alpha)   =  (A\ast B_i)(\alpha) ,
$$
where 
$
Q_\tau(\alpha)=\frac{1-E(\alpha^\tau)}{1-\alpha}.
$
\end{pro}
\begin{proof}
Fix the pattern $B_i$ $(1\le i \le m)$.
Let the gambling team places its bets on the occurrence of $B_i$ according to the rules  described in Section \ref{Gteam}. Let $X^i_n$ denote the total net gain of the casino 
at the moment $n$. If the initial word $A$ of the length $l$ is known then 
$$
X^i_l=\frac{1-\alpha^l}{1-\alpha}-\alpha^l(A\ast B_i)(\alpha).
$$
Define  $X^i_{l+n}$ as $Y^i_n$. The process $(Y^i_n)_{n=0}^\infty$ forms a bounded martingale. Let now the stopping time $\tau$ equal $\tau_{A\mathfrak{B}}$.
Observe  that 
a net gain of the casino depends on the case in which the  pattern $B_j$ $(1\le j \le m)$  is observed at time $\tau$. 
On the set $\{\tau=\tau_j\}$ it takes the value $\alpha^{l+\tau}
(B_j\ast B_i)(\alpha)$.
Thus for $\tau=\tau_{A\mathfrak{B}}$ we get
$$
Y^i_\tau=X^i_{l+\tau}=\frac{1-\alpha^{l+\tau}}{1-\alpha}-\alpha^{l+\tau}\sum_{j=1}^m(B_j\ast B_i)(\alpha){\bf 1}_{\{\tau=\tau_j\}}.
$$
By the optional stopping theorem we obtain the following equation
\begin{equation*}
\frac{1-\alpha^l}{1-\alpha}-\alpha^l(A\ast B_i)(\alpha)  = 
\frac{1-\alpha^lE(\alpha^\tau)}{1-\alpha}-\alpha^{l}
\sum_{j=1}^m(B_j\ast B_i)(\alpha)E(\alpha^\tau{\bf 1}_{\{\tau=\tau_j\}})
\end{equation*}
which, simplified and expressed in terms of  the functions $Q_\tau$ and $g^j_\tau$, completes the proof.
\end{proof}
\begin{rem}
The above Proposition contains a system of equations which is equivalent to (39) \cite{GeLi}. To derive this system Gerber and Li 
used a Markov chain embedding. This is a general method and of independent interest. In our proof of Proposition \ref{gamb} we show how immediately, by using  the classical gambling team technique for generating functions, this system can be obtained.
\end{rem}

\section{Generating functions for waiting time distributions}
Recall that the probability generating function for waiting time distribution $\tau$ equals
$$
g_\tau(\alpha)=E(\alpha^\tau)=\sum_{n=0}^\infty Pr(\tau=n)\alpha^n
$$
and coefficients of $\alpha^n$ in the power series $E(\alpha^\tau{\bf 1}_{\{\tau=\tau_j\}})$ are the probabilities $Pr(\tau=\tau_j=n)$ that is
$$
g_\tau^j(\alpha)=E(\alpha^\tau{\bf 1}_{\{\tau=\tau_j\}})=\sum_{n=0}^\infty Pr(\tau=\tau_j=n)\alpha^n.
$$
Taking into account that 
$$
g_\tau(\alpha)=E(\alpha^\tau)=\sum_{j=1}^m E(\alpha^\tau{\bf 1}_{\{\tau=\tau_j\}})=\sum_{j=1}^m g^j_\tau(\alpha),
$$
by Proposition \ref{gamb},
we obtain the following system of linear equations
\begin{equation}
\label{eq1}
\left\{
\begin{array}{ccl}
(1-\alpha)Q_\tau(\alpha)+\sum_{j=1}^m g_\tau^{j}(\alpha) & = & 1\\
-Q_\tau(\alpha)+\sum_{j=1}^m (B_j\ast B_i)(\alpha)g_\tau^{j}(\alpha)  & = & (A\ast B_i)(\alpha)\quad (1\le i \le m).
\end{array}
\right.
\end{equation}

Let $\mathcal{A}$ denote the coefficient matrix of the above system, i.e.

$$
\mathcal{A}(\alpha)=
\left(\begin{array}{c|c}
1-\alpha& \begin{array}{ccc} 1 & \ldots & 1\end{array} \\
\hline
\begin{array}{c} -1 \\ \vdots \\ -1 
\end{array} & (B_j\ast B_i)(\alpha)
\end{array}\right)_{1\le i,j \le m},
$$
$\mathcal{B}$ a matrix formed from the correlations functions $B_j\ast B_i$ that is
\begin{equation*}
\mathcal{B}(\alpha)=
\begin{pmatrix}
(B_j\ast B_i)(\alpha)
\end{pmatrix}
_{1\le i,j \le m},
\end{equation*}
and $\mathcal{B}^j$  stands for the matrix that arises on replacing the $j$th column of $\mathcal{B}$ by  the column vector of units 
$(1)_{1\le i \le m}$.
Using the Laplace expansion along the first rows and permuting  columns one can show that
$$
\det \mathcal{A}(\alpha)=(1-\alpha)\det \mathcal{B}(\alpha) +\sum_{j=1}^m \det \mathcal{B}^j(\alpha).
$$

Let $\mathcal{A}_{1+k}$ ($0\le k \le m$) denote the matrix formed by replacing the $1+k$th column of $\mathcal{A}$ by  the column vector  
$(1,(A\ast B_1),...,(A\ast B_m))$ and
$\mathcal{B}_k$ ($1\le k \le m$) be the matrices formed by replacing the $k$th column of $\mathcal{B}$ by  the column vector  
$((A\ast B_i))_{1\le i \le m}$.
Observe that similarly to above one gets
$$
\det \mathcal{A}_1(\alpha)=\det \mathcal{B}(\alpha) -\sum_{k=1}^m \det \mathcal{B}_k(\alpha)
$$
and   
$$
\det \mathcal{A}_{1+k}(\alpha)=(1-\alpha)\det \mathcal{B}_k(\alpha)+\sum_{j=1}^m \det \mathcal{B}_k^j(\alpha) \quad {\rm for}\; 1\le k \le m,
$$
where $\mathcal{B}_k^j$ is the matrix arising upon replacing the $j$th column of $\mathcal{B}_k$ by  the column vector of units 
$(1)_{1\le i \le m}$.

By the Cram\'er rules we obtain the following:
\begin{pro}
The solution of the system (\ref{eq1}) has the form
$$
Q_\tau(\alpha)=\frac{\det \mathcal{B}(\alpha) -\sum_{k=1}^m \det \mathcal{B}_k(\alpha)}{(1-\alpha)\det \mathcal{B}(\alpha) +\sum_{j=1}^m \det \mathcal{B}^j(\alpha)}
$$
and
$$
g_\tau^k(\alpha)=\frac{(1-\alpha)\det \mathcal{B}_k(\alpha)+\sum_{j=1}^m \det \mathcal{B}_k^j(\alpha)}{(1-\alpha)\det \mathcal{B}(\alpha) +\sum_{j=1}^m \det \mathcal{B}^j(\alpha)}
$$
for $1\le k \le m$.
\end{pro}
\begin{rem}
Let us emphasize that in the case where for each $i$ ($1\le i \le m$) $A\ast B_i\equiv 0$, which holds for instance when there is no initial pattern ($A=\emptyset$),
each matrix $\mathcal{B}_k$ possesses a column of zeros, it follows that $\det \mathcal{B}_k\equiv 0$ for $1\le k \le m$ and $\det \mathcal{B}_k^j\equiv 0$ but only
for $j\neq k$. If  $j=k$ then $\mathcal{B}_k^k=\mathcal{B}^k$.
Thus in this case
the above formulas take the simpler forms
$$
Q_\tau(\alpha)=\frac{\det \mathcal{B}(\alpha)}{(1-\alpha)\det \mathcal{B}(\alpha) +\sum_{j=1}^m \det \mathcal{B}^j(\alpha)}
$$
and
$$
g_\tau^{k}(\alpha)=\frac{\det \mathcal{B}^k(\alpha)}{(1-\alpha)\det \mathcal{B}(\alpha) +\sum_{j=1}^m \det \mathcal{B}^j(\alpha)}.
$$

\end{rem}

Since $g_\tau^{k}(1)$ equals $Pr(\tau=\tau_k)$ we get the following:
\begin{cor}
\label{wn1}
Let $\tau=\tau_{A\mathfrak{B}}$ and $\tau_k=\tau_{AB_k}$.
The probability  that the pattern $B_k$ precedes all the remaining $m-1$ patterns is equal to 
\begin{equation}
\label{pro}
Pr(\tau=\tau_k)=\frac{\sum_{j=1}^m \det \mathcal{B}_k^j(1)}{\sum_{j=1}^m \det \mathcal{B}^j(1)}.
\end{equation}
\end{cor}
And since $E\tau=Q_\tau(1)$  we can formulate the following:
\begin{cor}
The expected waiting time  to any pattern equals
$$
E\tau=\frac{\det \mathcal{B}(1) -\sum_{k=1}^m \det \mathcal{B}_k(1)}{\sum_{j=1}^m \det \mathcal{B}^j(1)}.
$$
\end{cor}  
\begin{rem}
Let us note that the above two corollaries show a form of the solution of the system of equations (3.7) given by Li in \cite{Li}. 
Under the assumption that an  initial pattern is not known ($A=\emptyset$), they turn into the following forms
\begin{equation}
\label{withoutA}
Pr(\tau=\tau_k)=\frac{\det \mathcal{B}^k(1)}{\sum_{j=1}^m \det \mathcal{B}^j(1)}\quad
{\rm and}\quad 
E\tau=\frac{\det \mathcal{B}(1)}{\sum_{j=1}^m \det \mathcal{B}^j(1)}.
\end{equation}
Matrix equations for these solutions can be found in \cite{PozKull}.
\end{rem}

\section{An example}
Let a casino generate $(\xi_n)$ by  a symmetric coin tossing.
Let $H$ and $T$ denote the heads and tails of the coin. Then $\Sigma=\{H,T\}$ and
$$
Pr(\xi=H)= Pr(\xi=T)=\frac{1}{2}.
$$
Take three patterns $B_1=THH$, $B_2=HTH$ and $B_3=HHT$; $\mathfrak{B}=\{THH,\;HTH,\;HHT\}$.
From now on we will write
$p_i$ instead of $Pr(\tau=\tau_i)$ and omit the value $1$ in the notation $\mathcal{B}(1)$ and  $(B_j\ast B_i)(1)$ that is 
$\mathcal{B}=\mathcal{B}(1)$ and $B_j\ast B_i=(B_j\ast B_i)(1)$.

One can calculate that the matrix
$$
\mathcal{B}= \begin{pmatrix} B_j\ast B_i\end{pmatrix}_{1\leq i,j \leq 3}= 
\begin{pmatrix}
	8       &  4   & 2 \\
	2 &   10      & 4 \\
	6    &    2  & 8
\end{pmatrix}.
$$
Consider any initial  pattern $A$ of the length $l\ge 3$ such that $B_i$ ($i=1,2,3$) is not a subpattern of $a_1...a_{l-1}$.
If the final subpattern $A^{(3)}=B_k$ for some $k=1,2,3$ then $A\ast B_i=B_k\ast B_i$ ($i=1,2,3$) and the matrix $\mathcal{B}_k=\mathcal{B}$.
By (\ref{pro}) one gets that $p_k=1$ and it follows that $p_i=0$ for $i\neq k$.

Assume now that $A^{(3)}\notin\mathfrak{B}$. Then $A\ast B_i=A^{(2)}\ast B_i$. We  consider in turn  four possible cases of $A^{(2)}$:
$HH,\;HT,\;TH,\;TT$. 

Let $A^{(2)}=HH$. Observe that since $A^{(3)}\notin\mathfrak{B}$ in particular $A^{(3)}\neq THH$, it follows $A^{(3)}=HHH$. Moreover 
because $THH$ is not subpattern of $A$ then $A$ must be a run of $H$. If now $\xi_1=H$ then the situation does not change  and we still have  the run of $H$ but if $T$ appears  then it finishes the game. It follows that $p_3=1$ and the other probabilities equal zero.

Let us check the above observations  by applying Corollary \ref{wn1}. For  $A^{(3)}\notin\mathfrak{B}$ and $A^{(2)}=HH$  we have
$$
A\ast B_1=0,\quad A\ast B_2=2\quad {\rm and} \quad A\ast B_3=6
$$
and matrices $\mathcal{B}_k$ ($k=1,2,3$) are equal
$$
\mathcal{B}_1=
\begin{pmatrix}
	0       &  4   & 2 \\
	2 &   10      & 4 \\
	6    &    2  & 8
\end{pmatrix},\quad 
\mathcal{B}_2=
\begin{pmatrix}
	8       &  0   & 2 \\
	2 &   2      & 4 \\
	6    &    6  & 8
\end{pmatrix}\quad
{\rm and}
\quad
\mathcal{B}_3=
\begin{pmatrix}
	8       &  4   & 0 \\
	2 &   10      & 2 \\
	6    &    2  & 6
\end{pmatrix}. 
$$
For the matrix $\mathcal{B}_1$ we get
$$
\sum_{j=1}^m \det \mathcal{B}_1^j=\det\begin{pmatrix}
	1       &  4   & 2 \\
	1 &   10      & 4 \\
	1    &    2  & 8
\end{pmatrix}
+ \det\begin{pmatrix}
	0       &  1   & 2 \\
	2 &   1  & 4 \\
	6    &    1  & 8
\end{pmatrix}
+ \det\begin{pmatrix}
	0       &  4   & 1 \\
	2 &   10  & 1 \\
	6    &    2  & 1
\end{pmatrix}
=40+0-40=0.
$$
In the same manner one can calculate that $\sum_{j=1}^3 \det \mathcal{B}_2^j=0$ and $\sum_{j=1}^3 \det \mathcal{B}_3^j=96$. Since
$\sum_{j=1}^3 \det \mathcal{B}^j =96$, by Corollary \ref{wn1} we do indeed obtain the confirmation of our previous observations:
$p_1=0$, $p_2=0$ and $p_3=1$.

In the case $A^{(2)}=HT$ ($A^{(3)}\notin \mathfrak{B}$) 
we have 
$$
A\ast B_1=2,\quad A\ast B_2=4\quad {\rm and} \quad A\ast B_3=0.
$$
Now we can form  matrices $\mathcal{B}_k$ and calculate $\sum_{j=1}^3 \det \mathcal{B}_1^j=32$, $\sum_{j=1}^3 \det \mathcal{B}_2^j=64$
and $\sum_{j=1}^3 \det \mathcal{B}_3^j=0$. Since $\sum_{j=1}^3 \det \mathcal{B}^j =96$, by Corollary \ref{wn1} we get 
$p_1=\frac{1}{3}$, $p_2=\frac{2}{3}$ and $p_3=0$. In this case it is not quite so easy to observe that the game may not finish the pattern
$B_3$.

For $A^{(2)}=TH$ ($A^{(3)}\notin \mathfrak{B}$) 
$$
A\ast B_1=4,\quad A\ast B_2=2\quad {\rm and} \quad A\ast B_3=2
$$
and the probabilities $p_i$ equal $p_1=\frac{2}{3}$, $p_2=\frac{1}{3}$ and $p_3=0$. Because for $A^{(2)}=TT$ 
and $A=T$ the values
$$
A\ast B_1=2,\quad A\ast B_2=0,\quad  A\ast B_3=0
$$
are the same then we obtain the same values of the probabilities: $p_1=\frac{2}{3}$, $p_2=\frac{1}{3}$ and $p_3=0$. Let us emphasize 
that the above values $A\ast B_i$ are different than in the previous case $A^{(2)}=TH$ but we obtained  the same values of probabilities $p_i$.

For completeness of presentation we should calculate the probabilities $p_i$  in the cases $A=H$ and $A=\emptyset$.
In the first one we get $p_1=\frac{1}{6}$, $p_2=\frac{1}{3}$ and $p_3=\frac{1}{2}$ and when there is no initial pattern $A$ by  (\ref{withoutA})
one gets $p_1=\frac{5}{12}$, $p_2=\frac{1}{3}$ and $p_3=\frac{1}{4}$. 

Thus we obtain the full description   of the probabilities  $Pr(\tau_{A\mathfrak{B}}=\tau_{AB_i})$ ($i=1,2,3$) for the given collection of patterns $\mathfrak{B}$ and  any initial word $A$.


\end{document}